\documentclass[12pt]{article}
\usepackage[top=3cm, bottom=3cm, left=2cm, right=2cm]{geometry}
\usepackage[alphabetic,lite]{amsrefs}                                          
\usepackage{amssymb}                                                           
\usepackage{amsthm}                                                            
\usepackage{appendix}                                                          
\usepackage{array}                                                             
\usepackage{bm}                                                                
\usepackage{colonequals}                                                       
\usepackage{enumitem}                                                          
\usepackage{fullpage}                                                          
\usepackage{graphicx}                                                          
\usepackage{indentfirst}                                                       
\usepackage{mathrsfs}                                                          
\usepackage{mathtools}                                                         
\usepackage{moreenum}                                                          
\usepackage{multirow}
\usepackage{pdfpages}
\usepackage{stmaryrd}                                                          
\usepackage{titling}
\usepackage{verbatim}                                                          
\usepackage[all,cmtip]{xy}                                                     
\usepackage{tikz}
\usepackage{xcolor}                                                            

\definecolor{tianred}{rgb}{0.57, 0.36, 0.51}                                   
\definecolor{tianblue}{rgb}{0.0, 0.22, 0.66}                                   
\definecolor{tianpink}{rgb}{0.88, 0.56, 0.59}                                  
\definecolor{tiangreen}{rgb}{0.24, 0.82, 0.44}                                 
\definecolor{Prune}{RGB}{99,0,60}

\usepackage[colorlinks,
            linkcolor=tianred,
            anchorcolor=tiangreen,
            citecolor=tianblue,
            urlcolor=tianpink]{hyperref}                                      
\usepackage{cleveref}                                                          

\makeatletter
\DeclareRobustCommand\widecheck[1]{{\mathpalette\@widecheck{#1}}}
\def\@widecheck#1#2{%
    \setbox\z@\hbox{\m@th$#1#2$}%
    \setbox\tw@\hbox{\m@th$#1%
       \widehat{%
          \vrule\@width\z@\@height\ht\z@
          \vrule\@height\z@\@width\wd\z@}$}%
    \dp\tw@-\ht\z@
    \@tempdima\ht\z@ \advance\@tempdima2\ht\tw@ \divide\@tempdima\thr@@
    \setbox\tw@\hbox{%
       \raise\@tempdima\hbox{\scalebox{1}[-1]{\lower\@tempdima\box
\tw@}}}%
    {\ooalign{\box\tw@ \cr \box\z@}}}
\makeatother

\newcommand{\beac}{\begin{equation}\begin{array}{c}}                           
\newcommand{\eeac}{\end{array}\end{equation}}                                  

\newcommand{\beq}{\begin{equation}}
\newcommand{\eeq}{\end{equation}}

\newcommand{\brmks}{\begin{rmk}\hfill\benuma}
\newcommand{\ermks}{\eenum\end{rmk}}

\newcommand{\benum}{\begin{enumerate}[label={{\upshape(\alph*)}}]}             
\newcommand{\benuma}{\begin{enumerate}[label={{\upshape(\arabic*)}}]}          
\newcommand{\benumr}{\begin{enumerate}[label={{\upshape(\roman*)}}]}           
\newcommand{\eenum}{\end{enumerate}}
\newcommand{\bitem}{\begin{itemize}}                                           
\newcommand{\eitem}{\end{itemize}}                                             

\newcommand{\xm}{\xymatrix}                                                    

\newcommand{\tstsl}{\tst\sum\limits}                                           
\newcommand{\tstpl}{\tst\prod\limits}                                          

\theoremstyle{plain}      \newtheorem{thm}{Theorem}[section]                   %
\theoremstyle{plain}      \newtheorem{lem}[thm]{Lemma}                         %
\theoremstyle{plain}      \newtheorem{cor}[thm]{Corollary}                     %
\theoremstyle{plain}      \newtheorem{prop}[thm]{Proposition}                  %
\theoremstyle{plain}      \newtheorem{conjecture}[thm]{Conjecture}             %
\theoremstyle{definition} \newtheorem{rmk}[thm]{Remark}                        %
\theoremstyle{definition} \newtheorem{df}[thm]{Definition}                     %
\theoremstyle{definition}                       %
\theoremstyle{definition}                         %
\theoremstyle{definition}                        %
\theoremstyle{definition}                    %
\theoremstyle{definition}                        %
\theoremstyle{definition}                    %
\theoremstyle{definition}                  %
\theoremstyle{definition}          %
\theoremstyle{definition}              %
\theoremstyle{definition}                  %
\theoremstyle{definition} \newtheorem{prop-df}[thm]{Proposition-Definition}    %

\theoremstyle{plain}                   %
\theoremstyle{plain}                             %
\theoremstyle{plain}                         %
\theoremstyle{definition}                      %
\theoremstyle{definition}                   %
\theoremstyle{definition}                        %
\theoremstyle{definition}                      %
\theoremstyle{definition}                  %
\theoremstyle{definition}                       %

\theoremstyle{definition}
\newtheorem*{construction*}{Construction}                                      %
\newtheorem*{conjecture*}{Conjecture}                                          %
\newtheorem*{hypothesis*}{Hypothesis}                                          %
\newtheorem*{convention*}{Convention}                                          %
\newtheorem*{notation*}{Notation}                                              %
\newtheorem*{prop*}{Proposition}                                               %
\newtheorem*{summary*}{Summary}                                                %
\newtheorem*{qt*}{Question}                                                    %
\newtheorem*{rmk*}{Remark}                                                     %
\newtheorem*{fact*}{Fact}                                                      %
\newtheorem*{lizi*}{Example}                                                   %
\newtheorem*{df*}{Definition}                                                  %
\theoremstyle{plain}
\newtheorem*{thm*}{Theorem}                                                    %
\newtheorem*{lem*}{Lemma}                                                      %
\newtheorem*{axiom*}{Axiom}                                                    %

\newtheoremstyle{subsection-tweak}
{}{}{}{}{\bfseries}{}{.5em}{\thmnumber{(\@{#1}{}\@{#2}).}\thmnote{~{\bfseries#3.}}}
\theoremstyle{subsection-tweak}
\newtheorem{para}[thm]{}

\crefname{thm}{\textup{Theorem}}{Theorems}                                     %

\crefname{lem}{\textup{Lemma}}{Lemmas}
\crefname{eg}{\textup{Example}}{Examples}
\crefname{ex}{\textup{Exercise}}{Exercise}
\crefname{rmk}{\textup{Remark}}{Remarks}
\crefname{cor}{\textup{Corollary}}{Corollaries}
\crefname{df}{\textup{Definition}}{Definitions}
\crefname{question}{\textup{Question}}{Questions}
\crefname{prop}{\textup{Proposition}}{Propositions}
\crefname{conjecture}{\textup{Conjecture}}{Conjectures}
\crefname{convention}{\textup{Convention}}{Conventions}
\crefname{construction}{\textup{Construction}}{Constructions}
\crefname{hypo}{\textup{Hypothesis}}{Hypothesis}

\crefname{thme}{Th\'eo\`eme}{Th\'eo\`emes}
\crefname{lemme}{Lemme}{Lemmes}
\crefname{ege}{Exemple}{Exemples}
\crefname{core}{Corollaire}{Corollaires}
\crefname{rmke}{Remarque}{Remarques}
\crefname{dfe}{D\'efinition}{D\'efinitions}

\newcommand{\bref}[1]{\textup{(\ref{#1})}}

\newcommand{\al}{\alpha}                                                       
\newcommand{\sg}{\sigma}                                                       
\newcommand{\OG}{\Omega}                                                       
\newcommand{\LB}{\Lambda}                                                      

\usepackage[OT2,T1]{fontenc}
\DeclareSymbolFont{cyrletters}{OT2}{wncyr}{m}{n}
\DeclareMathSymbol{\RBe}{\mathalpha}{cyrletters}{"42}                          
\DeclareMathSymbol{\Che}{\mathalpha}{cyrletters}{"51}                          
\DeclareMathSymbol{\Sha}{\mathalpha}{cyrletters}{"58}                          

\newcommand{\f}{\mathfrak}

\newcommand{\mrm}{\mathrm}

\newcommand{\BA}{\mathbf{A}}

   \newcommand{\BX}{\mathbf{X}}

\newcommand{\BBA}{\mathbb{A}}

\newcommand{\BBG}{\mathbb{G}}

\newcommand{\ui}{^{-1}}                                                        

\newcommand{\ol}{\overline}                                                    
\newcommand{\wt}{\widetilde}                                                   

\newcommand{\QZ}{\mathbb{Q}/\mathbb{Z}}                                        

\newcommand{\olk}{\overline{k}}                                                

\DeclareMathOperator{\gal}{Gal}                                                

\DeclareMathOperator{\id}{id}                                                  

\DeclareMathOperator{\tz}{char}                                                

\DeclareMathOperator{\ops}{\oplus}                                             

\DeclareMathOperator{\e}{Spec}                                                 

\DeclareMathOperator{\Ker}{Ker}                                                
\renewcommand{\ker}{\Ker}                                                      
\DeclareMathOperator{\Image}{Im}                                               
\renewcommand{\Im}{\Image}                                                     

\DeclareMathOperator{\Pic}{Pic}                                                
\DeclareMathOperator{\Br}{Br}                                                  

\newcommand{\gm}{\BBG_m}                                                       

\DeclareMathOperator{\rad}{rad}                                                

\newcommand{\Gss}{G^{\mathrm{ss}}}                                             
\newcommand{\Gtor}{G^{\mathrm{tor}}}                                           

\newcommand{\et}{\mathrm{\acute{e}t}}                                          

\newcommand{\nr}{{\mathrm{nr}}}                                                

\newcommand{\het}{H_{\et}}                                                     

\newcommand{\Brnr}{\Br_{\nr}}                                                  

\DeclareMathOperator{\ab}{ab}                                                  

\newcommand{\stra}[1]{\stackrel{#1}{\to}}                                      
\newcommand{\mt}{\mapsto}                                                      

\newcommand{\qand}{\quad\ \textup{and}\quad\ }                                 
\newcommand{\zc}{\,|\,}                                                        

\newcommand{\es}{\varnothing}                                                  
\DeclareMathOperator{\pr}{pr}                                                  

\newcommand{\ce}{\colonequals}                                                 

\newcommand{\inpart}{In particular,~}                                          
\newcommand{\wrt}{with respect to~}                                            

\newcommand{\itm}{\item}                                                       
\newcommand{\tst}{\textstyle}                                                  

\begin{document}

\title{A Simpler Approach to a Descent Conjecture of Wittenberg}

\author{Yisheng TIAN}

\maketitle

\begin{abstract}
A descent conjecture of Wittenberg \cite{Wit24}*{Conjecture 3.7.4} predicts that if all the twists of a rationally connected torsor over a smooth base satisfy weak approximation with Brauer--Manin obstruction,
then so does the base.
We give an alternative proof of Wittenberg's conjecture for certain torsors under connected linear groups via Cao's descent formula.
\end{abstract}

\section{Introduction}
The inverse Galois problem, asking whether any finite group $G$ is a quotient of $\gal(\ol{k}|k)$ for some number field $k$,
is a fundamental open question in number theory.
It has a positive answer when $G$ is symmetric or alternating
(Hilbert $1892$) and $G$ is solvable (Shafarevich $1954$).
Other classical results for sporadic groups or non-abelian simple groups of Lie type are summarized in \cite{Wit24}*{Section 1.1}.
As of today, there are many approaches to the inverse Galois problem as mentioned in \emph{loc. cit.}.
In what follows, we shall proceed by a descent method developed by Colliot-Th\'el\`ene--Sansuc, Harpaz--Wittenberg and others.

Throughout, $k$ is a number field.
Let $\OG$ be the set of all places of $k$ and let $k_v$ be the completion of $k$ at any $v\in\OG$.
By a $k$-variety $X$, we always mean a separated $k$-scheme of finite type.
Manin \cite{Man71} introduced a pairing (see \cite{Sko01}*{\S5.2} for more information)
\[
\tstpl_{v\in\OG}X(k_v)\times \Brnr(X)\to \QZ,\ ((x_v),\al)\mt \tstsl_{v\in\OG}j_v\circ x_v^*(\al),
\]
where
$\Brnr(X)$ is the unramified Brauer group of $X$ (see \cref{def: Brauer groups arith inv and nr} below),
$x_v^*:\Brnr(X)\to \Br(k_v)$ is the map induced by $x_v\in X(k_v)$ and
$j_v:\Br(k_v)\to \QZ$ is the local invariant.
Let $X(k_{\OG})^{\Brnr}$ be its left kernel which is a closed subset of $X(k_{\OG})\ce \prod_{v\in\OG}X(k_v)$
\wrt the product of $v$-adic topologies.
A programmatic conjecture is the following

\begin{conjecture}
[Colliot-Th\'el\`ene, \cite{CT03}*{p.~174}]\label{conj: CT on rationally conn varieties}
Let $X$ be a rationally connected smooth variety over $k$.
Then $X(k)$ is dense in $X(k_{\OG})^{\Brnr}$.
\end{conjecture}

This conjecture is interesting in its own right as it predicts the geometry of $X$ controls its arithmetic behavior.
Moreover, a significant consequence of this conjecture is the \emph{inverse Galois problem}.
More precisely,
by embedding a finite group $G$ into the symmetric group $\f{S}_n$ for some $n\ge 1$,
we may let $G$ act on $\BBA^n_k$ via the $\f{S}_n$-action.
Subsequently, taking the open subset $Y\subset \BBA_k^n$ consisting of points with pairwise distinct coordinates,
we obtain a so-called \emph{versal} $G$-torsor $Y\to Y/G$.
In applying \cref{conj: CT on rationally conn varieties} to $X\ce Y/G$,
we conclude a positive answer to the inverse Galois problem for $G$
(see \cite{Wit24}*{Section 3.4} for a detailed discussion).
From this perspective, the investigation of the inverse Galois problem may be transferred into that of \cref{conj: CT on rationally conn varieties},
in which the descent method plays a key role.

\begin{para}[Twisting of torsors]
Let $G$ be a linear algebraic group over $k$ and
let $Y\to X$ be a left $G$-torsor.
For any $[\sg]\in H^1(k,G)$, let $P{_{\sg}}\to X$ be a right $G$-torsor representing $[\sg]$.
We write ${_{\sg}}Y\ce P{_{\sg}}\times^GY$ for the contracted product.
Then the induced morphism ${_{\sg}}f:{_{\sg}}Y\to X$ is a ${_{\sg}}G$-torsor.
Moreover, there is a bijection $H^1(k,{_{\sg}}G)\to H^1(k,G)$ of sets that
sends the neutral class of $H^1(k,{_{\sg}}G)$ to $[\sg]\in H^1(k,G)$.
\end{para}

After introducing the twisting technique, the descent method leads us to the following

\begin{conjecture}
[Wittenberg, \cite{Wit24}*{Conjecture 3.7.4}]\label{conj: Wittenberg}
Let $X$ be a smooth $k$-variety and let $G$ be a linear algebraic $k$-group.
Let $f:Y\to X$ be a $G$-torsor with $Y$ rationally connected.
Assume that ${_{\sg}}Y(k)$ is dense in ${_{\sg}}Y(k_{\OG})^{\Brnr({_{\sg}}Y)}$ for any $[\sg]\in H^1(k,G)$.
Then $X(k)$ is dense in $X(k_{\OG})^{\Brnr(X)}$.
\end{conjecture}

If $G$ is a torus, the conjecture is known by the works \cites{CTS87b, HW20}.
In general, Linh \cite{Lin26} proved it for any connected linear group $G$ and any
rationally connected $k$-variety $X$
using the descent method developed by Harpaz--Wittenberg \cite{HW20} and
the abelianization machinery of Borovoi \cite{Bor98}.
In the present article, we give an alternative proof under weaker assumption based on the invariant Brauer subgroup introduced by Cao \cite{Cao18AF}.

\begin{thm}\label{thm: intro main thm}
Let $G$ be a connected linear group over $k$.
Let $X$ be a smooth geometrically integral $k$-variety.
Let $f:Y\to X$ be a $G$-torsor and let $Y^c$ be a smooth compactification of $Y$.
Assume that $\pi_1(Y^c_{\ol{k}})^{\ab}=0$ and $\Br(Y^c)/\Im \Br(k)$ is finite.
Then we have
\[
X(k_{\OG})^{\Brnr(X)}=\ol{\tst\bigcup\limits_{[\sg]\in H^1(k,G)} {_{\sg}}f({_{\sg}}Y(k_{\OG})^{\Brnr({_{\sg}}Y)})},
\]
where $\ol{\LB}$ denotes the closure of $\LB\subset X(k_{\OG})$ \wrt the product topology.
\end{thm}

\inpart if $Y$ is rationally connected, then the assumption of \cref{thm: intro main thm} on $Y$ is fulfilled
(see the proof of \cref{cor: Wittenberg conj via birational argument}).
The main ingredients of the proof is the following descent formula of Cao \cite{Cao18AF}*{Th\'eor\`eme 5.9}
\[
X(\BA)^{A}=\tst\bigcup\limits_{[\sg]\in H^1(k,G)}{_{\sg}}f({_{\sg}}Y(\BA)^{B_{\sg}+{_{\sg}}f^*(A)}),
\]
where $\BA$ is the ring of ad\`eles of $k$,
$A\subset \Br(X)$ is a subgroup and $B_{\sg}\subset \Br({_{\sg}}Y)$ is a subgroup containing $\Im\Br(k)$.
Subsequently, take $B_{\sg}=\Brnr({_{\sg}}Y)$ and $A=(f^*)\ui(B_{e})$, where $[e]\in H^1(k,G)$ is the neutral class.
In applying Harari's formal lemma, we may identify $X(k_{\OG})^{\Brnr(X)}$ with $\ol{X(\BA)^A}$
which yields the desired formula.

\begin{para}
[Acknowledgement]
The author is grateful to Yang CAO for many insightful discussions and helpful comments.
This article is supported by the grant of National Natural Science Foundation of China (no. 12401014).
\end{para}

\section{The invariant Brauer subgroup}
Let $k$ be a number field.
Let $k_{\OG}\ce \prod_{v\in \OG} k_v$ and
let $\BA$ be the ring of ad\`eles of $k$.
Let $G$ be a connected linear group over $k$.
A $k$-variety is a separated $k$-scheme of finite type.

\begin{df}\label{def: Brauer groups arith inv and nr}
Let $X$ be a smooth geometrically integral $k$-variety.
\benuma
\itm
The Brauer group of $X$ is defined as $\Br(X)\ce \het^2(X,\gm)$.
The {arithmetic Brauer group of $X$} is
\[
\Br_a(X)\ce \frac{\ker\big(\Br(X)\to \Br(X\times_k\ol{k})\big)}{\Im\big(\Br(k)\to \Br(X)\big)}.
\]

\itm
Suppose $X(k)\ne\es$.
Take any $x\in X(k)$ and let $x^*:\Br(X)\to \Br(k)$ be the induced homomorphism.
We put
\[
\Br_x(X)\ce \ker\big(\Br(X)\stra{x^*}\Br(k)\big).
\]
\inpart if we denote by $e$ the neutral element of $G$,
then $\Br_e(G)$ is defined.

\itm
Let $\rho:G\times_kX\to X$ be a left $G$-action on $X$.
After \cites{Cao18AF, Cao20},
the {\emph{invariant Brauer subgroup} of $\Br(X)$} is defined to be
\[
\Br_G(X)\ce \big\{b\in \Br(X)\zc (\rho^*(b)-p_2^*(b))\in p_1^*\Br(G)\big\},
\]
where $p_1:G\times_kX\to G$ and $p_2:G\times_kX\to X$ are the canonical projections.

\itm
The Brauer group is a birational invariant for smooth proper $k$-varieties by \cite{CTS21}*{Corollary 5.2.6}.
So we are allowed to define
the unramified Brauer group $\Brnr(X)$ of $X$ to be the Brauer group of
any smooth proper $k$-variety birationally equivalent to $X$.
\eenum
\end{df}

\begin{prop}\label{prop: pi1 vanish implies inv Br is Br}
Let $X$ be a smooth geometrically integral $k$-variety endowed with a left $G$-action.
If $\pi_1(X_{\olk})^{\ab}=0$,
then $\Br_G(X)=\Br(X)$.
\end{prop}
\begin{proof}
Since $\pi_1(X_{\olk})^{\ab}=0$,
we deduce $H^1(X_{\olk},\mu_n)=0$ by \cite{SGA1}*{Expos\'e XI, \S5, ($\ast$)}.
By \cite{Cao23}*{Th\'eor\`eme 2.1}, we obtain a canonical isomorphism of abelian groups
\[
(p_1^*,p_2^*):H^2(G_{\olk},\mu_n)\ops H^2(X_{\olk},\mu_n)\to H^2(G_{\olk}\times_{\olk} X_{\olk},\mu_n).
\]
The Hochschild--Serre spectral sequence $H^i(k,H^j(-_{\olk},\mu_n))\Rightarrow H^{i+j}(-,\mu_n)$
together with the $7$-term exact sequence in low degrees
yields an isomorphism of abelian groups
\[
(p_1^*,p_2^*):H^2_e(G,\mu_n)\ops H^2(X,\mu_n)\to H^2(G\times_{k} X,\mu_n).
\]
The Kummer exact sequence $0\to \mu_n\to \gm\to \gm\to 0$ then implies the surjectivity of the induced map
\[
(p_1^*,p_2^*):\Br_e(G)\ops \Br(X)\to \Br(G\times_k X).
\]
Observe that $i=(e,\id):\e k\times_k X\to G\times_k X$ induces a map
$i^*:\Br(G\times_kX)\to \Br(X)$ such that
$i^*\circ p_1^*(\Br_e(G))=0$ and $i^*\circ p_2^*=\id$.
So we have $\Im p_1^*\subset \ker i^*$.
Conversely, take any $\al\in \ker i^*\subset \Br(G\times_kX)$ and
suppose $\al=p_1^*g+p_2^*x$ for some $(g,x)\in \Br_e(G)\oplus \Br(X)$.
Then $x=i^*\circ p_2^*(x)=i^*\circ(p_1^*,p_2^*)(g,x)=i^*(\al)=0$ implies
$\al=p_1^*g$, i.e., $\ker i^*\subset \Im p_1^*$.

Now take any $b\in \Br(X)$.
Since $\rho\circ i=p_2\circ i=\id_X$,
we conclude $i^*(\rho^*b-p_2^*b)=0$ which implies
$\rho^*b-p_2^*b\in \ker i^*=\Im p_1^*$.
This shows $b\in \Br_G(X)$, i.e., $\Br(X)\subset \Br_G(X)$.
\end{proof}

\section{A descent formula of Cao}

\begin{para}\label{para: def of Theta}
Let $G$ be a connected linear group over $k$.
Let $Z$ be a smooth geometrically integral $k$-variety endowed with a left $G$-action.
For any $[\sg]\in H^1(k,G)$, let $P_{\sg}$ be a right $G$-torsor over $k$ representing $[\sg]$.
Consider the contracted product and the projection
\[
{_{\sg}}Z\ce P_{\sg}\times_k^GZ
\qand
\theta^{\sg}_Z:P_{\sg}\times_kZ\to {_{\sg}}Z.
\]
By \cite{Cao18AF}*{Lemme 3.12}, the projections
$p_1:{_{\sg}}Z\times_kZ\to {_{\sg}}Z$ and $p_2:{_{\sg}}Z\times_kZ\to Z$ induce a canonical \textbf{isomorphism} of abelian groups
\[
(p_1^*,p_2^*): \Br_a({_{\sg}}Z)\ops \frac{\Br_G(Z)}{\Im \Br(k)}\to
\frac{\Br_{{_{\sg}}G\times_kG}({_{\sg}}Z\times_kZ)}{\Im \Br(k)},
\]
which induces a further canonical homomorphism of abelian groups
\[
\xm{
\displaystyle\Theta^{\sg}_Z:\frac{\Br_{_{\sg}G}({_{\sg}}Z)}{\Im\Br(k)}
\ar[r]^-{(\theta^{\sg}_Z)^*} &
\displaystyle\frac{\Br_{{_{\sg}}G\times_kG}({_{\sg}}P\times_kZ)}{\Im\Br(k)}
\ar[r]^-{\pr} &
\displaystyle\frac{\Br_G(Z)}{\Im\Br(k)}.
}
\]
\end{para}

\begin{lem}
[\cite{Cao18AF}*{Lemme 5.8}]
Let $X$ be a smooth geometrically integral $k$-variety and
let $f:Y\to X$ be a left $G$-torsor.
The map $\Theta^{\sg}_Y$ is an isomorphism for any $[\sg]\in H^1(k,G)$.
\end{lem}

Similarly, for any $[\tau]\in H^1(k,{_{\sg}}G)$, there is an isomorphism of abelian groups
\[
\Theta_{{_{\sg}}Y}^{\tau}:
\frac{\Br_{{_{\tau}}(_{\sg}G)}({_{\tau}}(_{\sg}Y))}{\Im\Br(k)}
\to
\frac{\Br_{_{\sg}G}({_{\sg}}Y)}{\Im\Br(k)}.
\]

\begin{para}\label{para: Theta composite twice}
For any subgroup $B_{\sg}\subset \Br_{{_{\sg}}G}({_{\sg}}Y)$ containing $\Im\Br(k)$,
let $\ol{B_{\sg}}\ce B_{\sg}\bmod{\Im\Br(k)}$ and
denote by $\wt{\Theta}_Y^{\sg}(B_{\sg})\subset \Br_G(Y)$ the preimage of $\Theta^{\sg}_Y(\ol{B_{\sg}})\subset \Br_G(Y)/\Im\Br(k)$.
If we denote by $[\sg']\in H^1(k,{_{\sg}}G)$ the class of the inverse torsor of $P_{\sg}$
(see \cite{Sko01}*{p.~20, Example 2}),
then $\wt{\Theta}_{{_{\sg}}Y}^{\sg'}\circ \wt{\Theta}_Y^{\sg}(B_{\sg})= B_{\sg}$ by \cite{Cao18AF}*{Lemme 5.8}.
\end{para}

\begin{thm}
[\cite{Cao18AF}*{Th\'eor\`eme 5.9}]\label{thm: descent formula from Cao18}
Let $X$ be a smooth geometrically integral $k$-variety and
let $f:Y\to X$ be a left $G$-torsor.
For any $[\sg]\in H^1(k,G)$,
let $B_{\sg}\subset \Br_{{_{\sg}}G}({_{\sg}}Y)$ be a subgroup containing $\Im\Br(k)$.
Let $A\subset \Br(X)$ be a subgroup such that for any $[\sg]\in H^1(k,G)$,
\beq\label{eq: complicated condition for descent formula}\tag{$\dagger$}
({_{\sg}}f^*)\ui\Big(\tstsl_{\tau\in\Sha^1({_{\sg}}G)}\wt{\Theta}_{{_{\sg}}Y}^{\tau}(B_{\sg+\tau})\Big)\subset A,
\eeq
where $B_{\sg+\tau}\subset \Br_{{_{\tau}}({_{\sg}}G)}({_{\tau}}({_{\sg}}Y))$.
Then the following descent formula holds
\[
X(\BA)^{A}=\tst\bigcup\limits_{[\sg]\in H^1(k,G)}{_{\sg}}f({_{\sg}}Y(\BA)^{B_{\sg}+{_{\sg}}f^*(A)}).
\]
\end{thm}

In the sequel, we only need the following case
where the condition \bref{eq: complicated condition for descent formula} is simplified.

\begin{lem}\label{lemma: descent formula from Cao18}
Let $A\subset \Br(X)$ be a subgroup such that
for each $[\sg]\in H^1(k,G)$
\[
(f^*)\ui(B)\subset A
\qand
\Theta^{\sg}_Y(\ol{B_{\sg}})\subset \ol{B},
\]
where $B\subset \Br_G(Y)$ is the subgroup for the neutral class $[e]\in H^1(k,G)$.
Then \bref{eq: complicated condition for descent formula} holds.
\end{lem}
\begin{proof}
For any $[\sg]\in H^1(k,G)$ and any $[\tau]\in H^1(k,{_{\sg}}G)$,
let $[\sg+\tau]\in H^1(k,G)$ denote the class corresponding to the neutral class of
$H^1(k,{_{\tau}}({_{\sg}}G))$.
By assumption, we conclude $\wt{\Theta}_{Y}^{\sg+\tau}(B_{\sg+\tau})\subset B\subset f^*(A)$.
Then applying $(\wt{\Theta}_Y^{\sg})\ui$ yields
$\wt{\Theta}_{_{\sg}Y}^{\tau}(B_{\sg+\tau})\subset {_{\sg}}f^*(A)$.
\end{proof}

\section{Wittenberg's descent conjecture}
\begin{para}\label{para: convention}
Throughout this section,
let $X$ be a smooth geometrically integral $k$-variety.
Let $G$ be a connected linear $k$-group and
let $f:Y\to X$ be a $G$-torsor.
\end{para}

The next lemma is probably well-known.
We still write it here for the lack of a reference.

\begin{lem}\label{lem: finiteness of Pic}
Keep the same notation as in Paragraph \textup{\bref{para: convention}}.
The Picard group $\Pic(G)$ and the kernel $\ker(\Br(X)\to \Br(Y))$ are finite.
\end{lem}
\begin{proof}
Let $\rad^u(G)$ be the unipotent radical of $G$ and
let $G^{\mrm{red}}\ce G/\rad^u(G)$.
The exact sequence $1\to \rad^u(G)\to G\to G^{\mrm{red}}\to 1$
induces an exact sequence $\Pic(G^{\mrm{red}})\to \Pic(G)\to \Pic(\rad^u(G))$ by \cite{San81}*{Corollaire 6.11}.
Since the underlying variety of $\rad^u(G)$ is affine,
the map $\Pic(G^{\mrm{red}})\to \Pic(G)$ is surjective.
Thus it suffices to show that $\Pic(G^{\mrm{red}})$ is finite.

Let $\Gss$ be the derived subgroup of $G^{\mrm{red}}$ (which is semi-simple) and
let $\Gtor\ce G^{\mrm{red}}/\Gss$ (which is a torus).
There is an exact sequence $\Pic(\Gtor)\to \Pic(G^{\mrm{red}})\to \Pic(\Gss)$ by \emph{loc. cit.}.
By \cite{San81}*{Lemme 6.9},
we see that $\Pic(\Gtor)\simeq H^1(k,\BX^*(\Gtor))$
(where $\BX^*(\Gtor)$ is the module of characters of $\Gtor$) and that $\Pic(\Gss)$ is finite.
But $H^1(k,\BX^*(\Gtor))$ is a finitely generated torsion abelian group,
so it is finite.
Consequently, $\Pic(G^{\mrm{red}})$ is finite as well.

Finally, due to \cite{San81}*{(6.10.1)} there is an exact sequence $\Pic(G)\to \Br(X)\to \Br(Y)$ of abelian groups.
Thus the finiteness of $\Pic(G)$ yields that of $\ker(\Br(X)\to \Br(Y))$.
\end{proof}

\begin{thm}\label{thm: main theorem of Wittenberg conj}
Keep the same notation as in Paragraph \textup{\bref{para: convention}}.
Let $Y^c$ be a smooth compactification of $Y$.
If $\pi_1(Y^c_{\ol{k}})^{\ab}=0$ and $\Br(Y^c)/\Im \Br(k)$ is finite,
then we have
\beq\label{eq: descent formula II}
X(k_{\OG})^{\Brnr(X)}=\ol{\tst\bigcup\limits_{[\sg]\in H^1(k,G)} {_{\sg}}f({_{\sg}}Y(k_{\OG})^{\Brnr({_{\sg}}Y)})}.
\eeq
If ${_{\sg}}Y(k)$ is dense in ${_{\sg}}Y(k_{\OG})^{\Brnr({_{\sg}}Y)}$ for any $[\sg]$,
then $X(k)$ is dense in $X(k_{\OG})^{\Brnr(X)}$.
\end{thm}
\begin{proof}
Recall that $\pi_1(Y^c_{\olk})$ and $\Br(Y^c)$ are birational invariants for smooth proper $k$-varieties
(see \cite{SGA1}*{Expos\'e X, Corollaire 3.4} and \cite{CTS21}*{Corollary 5.2.6} respectively),
so the assumptions on $Y^c$ are independent of the choice of it.
According to \cite{Bri22}*{Theorem 2},
we may choose $Y^c$ such that the $G$-action on $Y$ extends to it $G$-equivariantly.
Subsequently, the twist ${_{\sg}}Y^c$ is a ${_{\sg}}G$-equivariant smooth compactification of ${_{\sg}}Y$
for each $[\sg]\in H^1(k,G)$.
Let $\Theta_{Y^c}^{\sg}$ be the canonical homomorphism defined in \bref{para: def of Theta}
which makes the diagram commutative 
\beac\label{diag: main theorem comm diag}
\xm{
  {\Br_{{_{\sg}}G}({_{\sg}}Y^c)}/{\Im \Br(k)}\ar[r]\ar[d]_-{\Theta_{Y^c}^{\sg}}
& {\Br_{{_{\sg}}G}({_{\sg}}Y)}/{\Im \Br(k)}\ar[d]^-{\Theta_{Y}^{\sg}} \\
  {\Br_G(Y^c)}/{\Im \Br(k)}\ar[r]
& {\Br_G(Y)}/{\Im \Br(k)}.
}
\eeac
For each $[\sg]\in H^1(k,G)$,
let $B_{\sg}\ce \Brnr({_{\sg}}Y)$ and
$B\ce \Brnr(Y)$.
Let $f^*:\Br(X)\to \Br(Y)$ be the induced map and
let $A\ce (f^*)\ui(B)\subset \Br(X)$.
By assumption, we obtain $\pi_1({_{\sg}}Y^c_{\ol{k}})^{\ab}=0$  and
hence $\Br({_{\sg}}Y^c)=\Br_{{_{\sg}}G}({_{\sg}}Y^c)$ by \cref{prop: pi1 vanish implies inv Br is Br}.
So we have
\[
B_{\sg}\ce \Brnr({_{\sg}}Y)\simeq \Br({_{\sg}}Y^c)=\Br_{_{\sg}G}({_{\sg}}Y^c)\subset \Br_{_{\sg}G}({_{\sg}}Y).
\]
It follows that $\Theta^{\sg}_Y(\ol{B_{\sg}})=\Theta^{\sg}_{Y^c}(\ol{B_{\sg}})\subset \Br(Y^c)/\Im\Br(k)=\ol{B}$ by
\bref{diag: main theorem comm diag}.
Thus
$\Theta^{\sg}_Y(\ol{B_{\sg}})=\ol{B}$ by \bref{para: Theta composite twice}
and $A=(f^*)\ui(B)=({_{\sg}}f^*)\ui(B_{\sg})$.
\inpart the conditions of \cref{lemma: descent formula from Cao18} are fulfilled.
Consequently, we deduce
\[
X(\BA)^A
=\tst\bigcup\limits_{[\sg]\in H^1(k,G)}{_{\sg}}f({_{\sg}}Y(\BA)^{\Brnr({_{\sg}}Y)+{_{\sg}}f^*(A)})
=\tst\bigcup\limits_{[\sg]\in H^1(k,G)}{_{\sg}}f({_{\sg}}Y(\BA)^{\Brnr({_{\sg}}Y)}),
\]
where the last equality follows from ${_{\sg}}f^*(A)=\Brnr({_{\sg}}Y)$.

Thanks to \cref{lem: finiteness of Pic}, the groups $\Pic(G)$ and $\ker f^*$ are finite.
Since $\Br(Y^c)/\Im \Br(k)$ is finite by assumption,
the quotient $(f^*)\ui(\Br(Y^c))/\Im \Br(k)=A/\Im \Br(k)$ is also finite.
Since $\Brnr(X)\subset A$ by construction,
we conclude $\ol{X(\BA)^A}=X(k_{\OG})^{\Brnr(X)}$ by Harari's formal lemma \cite{Har94}*{Corollaire 2.6.1}
(see also \cite{Lin26}*{Lemma 3.10(ii)} for a detailed argument).
Subsequently, we immediately deduce
\[
\ol{\tst\bigcup\limits_{[\sg]\in H^1(k,G)} {_{\sg}}f({_{\sg}}Y(\BA)^{\Brnr({_{\sg}}Y)})}
=
\ol{\tst\bigcup\limits_{[\sg]\in H^1(k,G)} {_{\sg}}f({_{\sg}}Y(k_{\OG})^{\Brnr({_{\sg}}Y)})}
=X(k_{\OG})^{\Brnr(X)},
\]
where the first equality follows from Harari's formal lemma
together with the finiteness of $\Br({_{\sg}}Y^c)/\Im \Br(k)$.

The continuity of ${_{\sg}}f$ implies the density of ${_{\sg}}f({_{\sg}}Y(k))$ in ${_{\sg}}f({_{\sg}}Y(k_{\OG})^{\Brnr({_{\sg}}Y)})$.
Hence $X(k)$ is dense in
$\bigcup {_{\sg}}f({_{\sg}}Y(k_{\OG})^{\Brnr({_{\sg}}Y)})$ and
we obtain $\ol{X(k)}=X(k_{\OG})^{\Brnr(X)}$ by \bref{eq: descent formula II}.
\end{proof}

As an immediate consequence, we conclude the promised conjecture of Wittenberg.
The argument is probably well-known,
but we still give a complete proof for the convenient of the readers.

\begin{cor}\label{cor: Wittenberg conj via birational argument}
let $X$ be a smooth geometrically integral $k$-variety.
Let $G$ be a connected linear $k$-group and
let $f:Y\to X$ be a $G$-torsor with $Y$ rationally connected.
Assume that ${_{\sg}}Y(k)$ is dense in ${_{\sg}}Y(k_{\OG})^{\Brnr({_{\sg}}Y)}$ for any $[\sg]\in H^1(k,G)$.
Then $X(k)$ is dense in $X(k_{\OG})^{\Brnr(X)}$.
\end{cor}
\begin{proof}
Since $X$ is geometrically integral and $\tz(k)=0$,
the $G$-torsor $Y$ is also geometrically integral.
Let $Y^c$ be a smooth compactification of $Y$ which is irreducible.
Hence Chow's lemma \cite{EGAII}*{Th\'eor\`eme 5.6.1 and Corollaire 5.6.2} yields
a birational surjective morphism $Y'\to Y^c$ with projective $Y'$.
Let $Y''\to Y'$ be a resolution of singularities with projective $Y''$.
Now $Y''$ is a smooth projective variety that is birationally equivalent to $Y^c$,
thus $Y''$ is rationally connected.
So we deduce $\pi_1(Y^c_{\olk})\simeq \pi_1(Y''_{\olk})=0$
where the  vanishing follows from \cite{Kol96}*{p.~200, Proposition (3.3.1)} and \cite{Kol01}*{Theorem 13}.
Since the group $\Br(Y^c)/\Im \Br(k)$ 
is finite by \cite{CTS21}*{Corollary 4.4.4},
\cref{thm: main theorem of Wittenberg conj} implies the density of $X(k)$, i.e.,
\textup{\cref{conj: Wittenberg}} holds.
\end{proof}

\small
Yisheng TIAN

Institute for Advanced Study in Mathematics

Harbin Institute of Technology

Harbin 150001, China

Email: tysmath@mail.ustc.edu.cn

\end{document}